\documentclass[10pt]{amsart}
\usepackage{amscd, amsmath, amsthm, amssymb, multicol}
\usepackage[all]{xy}

\usepackage{amsthm}
\newtheorem{theorem}{Theorem}[section]
\newtheorem{lemma}[theorem]{Lemma}
\newtheorem{corollary}[theorem]{Corollary}
\newtheorem{proposition}[theorem]{Proposition}

\theoremstyle{definition}

\newtheorem{example}[theorem]{Example}

\theoremstyle{remark}
\newtheorem{remark}[theorem]{Remark}
\newtheorem{conjecture}[theorem]{Conjecture}

\numberwithin{equation}{section}

\newcommand{\calD}{\mathcal{D}}

\newcommand{\calI}{\mathcal{I}}

\newcommand{\calN}{\mathcal{N}}
\newcommand{\calO}{\mathcal{O}}

\newcommand{\calQ}{\mathcal{Q}}

\newcommand{\calS}{\mathcal{S}}

\newcommand{\calZ}{\mathcal{Z}}

\newcommand{\be}{\mathcal{\begin{equation}}}
\newcommand{\ee}{\mathcal{\end{equation}}}

\newcommand{\bbH}{\mathbb{H}}

\newcommand{\bbS}{\mathbb{S}}

\newcommand{\bbF}{\mathbb{F}}
\newcommand{\bbC}{\mathbb{C}}

\newcommand{\bbN}{\mathbb{N}}
\newcommand{\bbP}{\mathbb{P}}
\newcommand{\bbQ}{\mathbb{Q}}
\newcommand{\bbR}{\mathbb{R}}

\newcommand{\bbZ}{\mathbb{Z}}

\newcommand{\bfe}{\mathbf{e}}

\newcommand{\bfS}{\mathbf{S}}

\newcommand{\frakS}{\mathfrak{S}}

\newcommand{\bfO}{\mathbf{O}}

\newcommand{\Or}{\textup{O}}
\newcommand{\SL}{\textup{SL}}
\newcommand{\SO}{\textup{SO}}

\newcommand{\Ker}{\textup{Ker}}
\newcommand{\PGL}{\textup{PGL}}
\newcommand{\PSL}{\textup{PSL}}

\newcommand{\Sp}{\textup{Sp}}

\newcommand{\Proj}{\textup{Proj}}

\newcommand{\Jac}{\textup{Jac}}

\newcommand{\Sym}{\textup{Sym}}

\newcommand{\la}{\langle}
\newcommand{\ra}{\rangle}
\newcommand{\half}{\tfrac{1}{2}}

\newcommand{\discr}{\textup{discr}}

\newcommand{\PO}{\textup{PO}}
\newcommand{\Cay}{\textup{Cay}}

\newcommand{\beq}{\begin{equation}}
\newcommand{\eeq}{\end{equation}}

\date{}

\title{Configuration spaces of complex and real spheres}

\author{Igor Dolgachev}

\address{Department of Mathematics, University of Michigan, 525 E. University Av., Ann Arbor, MI, 49109}
\email{idolga@umich.edu}
\dedicatory{To Rob Lazarsfeld on the occasion of his 60th birthday}
\author{Benjamin  Howard}

\address{Center for Communications Research, Institute for Defense Analysis, 805 Bunn Dr., Princeton, NJ, 08540}
\email{bjhowa3@idaccr.org}
\begin{document}

\begin{abstract} We study the GIT-quotient of the Cartesian power of projective space modulo the projective orthogonal group. A classical isomorphism of this group with the Inversive group of birational transformations of the projective space of one dimension less allows one to interpret these spaces as configuration spaces of complex or real spheres. 
\end{abstract}

\maketitle

\maketitle
\section{Introduction}
In this paper we study the moduli space of configurations of points in complex projective space with respect to the group of projective transformations leaving invariant a non-degenerate quadric. More precisely, if $\bbP^n = \bbP(V)$ denotes the projective space of lines in  a linear complex space $V$ equipped with a non-degenerate symmetric form $\la v,w\ra$,  we study the GIT-quotient 
$$\bfO_n^m := \bbP(V)^m/\!/\PO(V) = \Proj \bigl(\bigoplus_{d=0}^\infty H^0(\bbP(V)^m,\calO_{\bbP(V)}(d)^{\boxtimes m}\bigr)^{\Or(V)}$$
$$ \cong \Proj \bigl(\bigoplus_{d=0}^\infty (S^d(V^*)^{\otimes m}\bigr)^{\Or(V)}.
$$
If $m \geq n+1 = \dim~V$ then generic point configurations have zero dimensional isotropy subgroups in $\Or(V)$, and since $\dim \Or(n+1) = \frac{1}{2}n(n+1)$ we expect that $\dim \bfO_n^m = mn-\frac{1}{2}n(n+1)$.

Let 
$$R(n;m)  = \bigoplus_{d=0}^\infty (S^d(V^*)^{\otimes m}\bigr)^{\Or(V)}.$$
It is a finitely generated graded algebra with graded part $R(n;m)_d$ of degree $d$ equal to $(S^d(V^*)^{\otimes m}\bigr)^{\Or(V)}$. After polarization, $R(n;m)_d $ becomes isomorphic to the linear space $\bbC[V^{m}]_{d,\ldots,d}^{\Or(V)}$ of $\Or(V)$-invariant polynomials on $V^m$ which are homogeneous of degree $d$ in each  vector variable.  The First Fundamental Theorem of Invariant theory  (FFT) for the orthogonal group \cite{Weyl}, Chapter 2, \S 9 asserts that $\bbC[V^{m}]^{\Or(V)}$ is generated by the bracket-functions $[ij]:(v_1,\ldots, v_m) \mapsto \la v_i,v_j\ra$. Using this theorem, our first result is the following.

\begin{theorem} Let $\Sym_{m}$ be the space of symmetric matrices of size $m$ with the torus $T^{m-1} = \{(z_1,\ldots,z_{m})\in (\bbC^*)^{m}:z_1\cdots z_{m} = 1\}$ acting by scaling each $i$-th row and $i$-th column by  $z_i$. Let $\bbS_{m}$ be the toric variety $\bbP(\Sym_{m})/\!/T^{m-1}$. Then $\bfO_n^m$ is isomorphic to a closed subvariety of $\bbS_{m}$ defined by the rank condition $r\le n+1$. 
\end{theorem}

For example, when $m \leq n+1$, we obtain that $\bfO_{n}^{m}$ is a toric variety of dimension $\frac{1}{2}n(n+1)$.

The varieties $\bfO_1^m$ are special since the connected component of the identity of $\Or(2)$ is isomorphic to $\SO(2) \cong \bbC^*$. This implies that $\bfO_1^m$ admits a double cover isomorphic to a toric variety $(\bbP^1)^m/\!/\SO(2)$.  We compare this variety with the toric variety $X(A_{m-1})$ associated to the root system of type $A_{m-1}$ (see \cite{DL}, \cite{Procesi1}). The variety $X(A_{m-1})$ admits a natural involution defined by the standard Cremona transformation of $\bbP^{m-1}$ and the quotient by this involution is a generalized Cayley 4-nodal cubic surface $\Cay_{m-1}$ (equal to the Cayley cubic surface if $m = 3$). We prove that $\bfO_1^m$ is isomorphic to $\Cay_{m-1}$ for odd $m$ and equal to some blow-down of $\Cay_{m-1}$ when $m$ is even.

The main geometric motivation for our work is the study of configuration spaces of complex and real spheres. It is known since F. Klein and S. Lie  that the Inversive group\footnote{Also called the Inversion group or the Laguerre group. It is a subgroup of the Cremona group of $\bbP^n$ generated by the projective affine orthogonal group $\textup{PAO}(n+1)$ and the inversion transformation $[x_0,\ldots,x_n]\mapsto [x_1^2+\ldots+x_n^2,x_0x_1,\ldots,x_0x_n]$.} defining the geometry of spheres  in dimension $n$ is isomorphic to the projective orthogonal group $\PO(n+1)$ (see, for example, \cite{Klein}, \S 25). Thus any problem about configurations of $m$ spheres in $\bbP^n$ is equivalent to the problem about configurations of $m$ points in $\bbP^{n+1}$ with respect to $\PO(n+1)$. The last two sections of the paper give some applications to the geometry of spheres.

\section{The First Fundamental Theorem of Invariant theory}

Let $V$ be an $n+1$-dimensional  quadratic complex vector space, i.e. a vector space together with a nondegenerate symmetric bilinear form whose values we denote by $\la v,w \ra$. Let $G = \Or(V)$ be the  orthogonal group of $V$ and  $\PO(V) = \Or(V)/\{\pm I\}$. Consider the diagonal action of  $G$ on $V^m$. The First Fundamental Theorem of Invariant theory for the orthogonal group (see \cite{Procesi}, Chapter 11, 2.1, \cite{Weyl}, Chapter 2, \S 9) asserts that any $G$-invariant polynomial function on $V^m$ is a polynomial in the bracket-functions 
$$[ij]:V^{m} \to \bbC, \ (v_1,\ldots,v_m)\mapsto  \la v_i,v_j\ra,  \quad 1\le i,j\le m.$$

The algebra of $G$-invariant polynomial functions $\bbC[V^m]^G$ has a natural multi-grading by $\bbN^m$ with homogeneous part $\bbC[V^m]_{(d_1,\ldots,d_m)}^G$ equal to the linear space  of polynomials which are homogeneous of degree $d_i$ in each  $i$-th vector variable. This grading corresponds to the natural action of the torus $\bbC^*{}^m$ by scaling the vectors in each factor. The $\bbN$-graded  ring $R(n;m)$ in which we are interested is the subring 
$\oplus_{d=0}^\infty \bbC[V^m]_{(d,\ldots,d)}^G$. 
We have
$$R(n;m) \cong \bbC[V^m]^{\Or(V)\times T},$$
where $T = \{(z_1,\ldots,z_m)\in \bbC^*{}^m:z_1\cdots z_m = 1\}$. 

Let $\Sym_m$ denote the linear space of complex symmetric $m\times m$-matrices.  If we view $V^{m}$ as the space of linear functions $L(\bbC^m,V)$, then we can define a quadratic map 
$$\Phi:V^{m}\to \Sym_m, \quad (v_1,\ldots,v_m)\mapsto (\la v_i,v_j\ra),$$
by composing 
$$\bbC^m\overset{\phi}{\longrightarrow} V \overset{b}{\longrightarrow}  V^*\overset{{}^t\phi}{\longrightarrow} (\bbC^m)^*,$$
where the middle map is defined by the symmetric bilinear form $b$ associated to $q$. It is easy to see that, considering the domain and the range of $\Phi$ as affine spaces over $\bbC$, the image of $\phi$ is the closed subvariety  $\Sym_m(n+1)\subset \Sym_m$ of symmetric matrices of rank $\le n+1$. Passing to the rings of regular functions, we get a homomorphism of rings
\beq\label{pol}
\Phi:\bbC[\Sym_m] \to \bbC[V^m].
\eeq
The map $\Phi$ is obviously $T$-equivariant if we make  $(z_1,\ldots,z_m)$ act by multiplying the entry $x_{ij}$ of a symmetric matrix by $z_iz_j$. By passing to invariants, we obtain a homomorphism of graded rings
\beq\label{pol1}
\Phi_T: \bbC[\Sym_m]^T\to \bbC[V^m]^{T}.
\eeq
The FTT can be restated by saying that the image of this homomorphism is equal to the ring $R(n;m)$. 

We identify $\bbC[\Sym_m]$ with the polynomial ring in entries $X_{ij}$ of a general symmetric  matrix $X = (X_{ij})$ of size $m\times m$. Note that the action of $(z_1,\ldots,z_m)\in (\bbC^*)^m$ on a symmetric matrix $(X_{ij})$ is by multiplying each entry $X_{ij}$ by $z_iz_j$.  The graded part $\bbC[\Sym_m]_d$ of $\bbC[\Sym_m]$ consists of functions which under this action are multiplied by $(z_1\cdots z_m)^d$. They are obviously contained in $\bbC[\Sym_m]^T$ and define the grading of the ring  $\bbC[\Sym_m]^T$. The homomorphism $\Phi^*$ is a homomorphism of  graded rings from  $\bbC[\Sym_m]^T$ to $R(n;m)$.

%Since we are interested in the projective spectrum of the graded ring $R(n;m)$ we can replace it by 
%$R(n,d)^{(2)}$ without changing the isomorphism class of the projective spectrum.  

 Let 
$$\det X = \sum_{\sigma\in \frakS_m}\epsilon(\sigma)X_{\sigma(1)1}\cdots X_{\sigma(N)N}$$
be the determinant of $X$. The monomials  $d_\sigma = X_{\sigma(1)1}\cdots X_{\sigma(N)N}$ will be called the \emph{determinantal} terms.  Note that the number $k(m)$ of different determinantal terms is less than $m!$.  It was known since the 19th century (\cite{Salmon}, p. 46) that the generating function for the numbers $k(m)$ is equal to 
$$1+\sum_{m=1}^\infty \frac{1}{m!}k(m)t^m = \frac{e^{\frac{1}{2}t+\frac{1}{4}t^2}}{\sqrt{1-t}}.$$
For example, $k(3) = 5, k(4) = 17, k(5) = 73, k(6) = 338.$

Each permutation $\sigma$ decomposes into disjoint oriented cycles.
Consider the directed graph on $m$ vertices which consists of the oriented cycles in $\sigma$; 
i.e. we take a directed edge $i \rightarrow \sigma(i)$ for each vertex $i$.
Suppose there is a cycle $\tau$ in $\sigma$ of length $\geq 3$.  
Write $\sigma = \tau \upsilon = \upsilon \tau$, and define $\sigma' = \tau^{-1} \upsilon$.  
Since our matrix is symmetric, the determinantal term $d_{\sigma'}$ 
corresponding to  $\sigma'$ has the same value as $d_\sigma$, and furthermore $\sigma'$ has the same sign 
as $\sigma$ (so there is no canceling), and so we may drop the orientation on 
each cycle.  We may therefore envision the determinantal terms as 
$2$-regular un-directed graphs on $m$ vertices (where $2$-cycles and loops are admitted).
Thus for each $2$-regular graph having $k$ cycles of length $\geq 3$, there corresponds $2^k$ 
determinantal terms.

\begin{proposition}\label{P3.1} The ring $\oplus_{d=0}^\infty \bbC[\Sym_m]_{2d}^T$ is generated by the determinantal terms.
\end{proposition}

\begin{proof} A monomial $X_{i_1j_1}\cdots X_{i_kj_k}$ belongs to $\bbC[\Sym_m]_{d}^T$ if and only if 
$$z_{i_1}\cdots z_{i_k}z_{j_1}\cdots z_{j_k} = (z_1\cdots z_m)^d$$ for any $z_1,\ldots,z_m\in \bbC^*$. This happens if and only if each $i\in \{1,\ldots,m\}$ occurs exactly $d$ times among $i_1,\ldots,i_k,j_1,\ldots,j_k$.  Consider the graph with set of  vertices equal to $\{1,\ldots,m\}$ and an edge from $i$ to $j$ if $X_{ij}$ enters into the monomial. The above property is equivalent to that the graph is a regular graph of valency $d$. The multiplication of monomials corresponds to the operation of adding graphs (in the sense that we add the sets of the edges). It remains to use  that any regular graph of valency $2d$ is equal to the union of regular graphs of valency $2$ (this is sometimes called a ``$2$-factorization'' or Petersen's factorization theorem) \cite{Petersen}, \S 9. 
\end{proof}

\begin{corollary} A set $([v_1],\ldots,[v_m])$ is semi-stable for the action of $\Or(V)$ on $(\bbP^n)^m$ if and only if there exists $\sigma\in \frakS_m$ such that $
\la v_{\sigma(1)},v_1\ra \cdots \la v_{\sigma(m)},v_m\ra$ is not equal to zero. 
\end{corollary} 

We can make it more explicit. 

\begin{proposition}\label{P3.3} A point set $([v_1],\ldots,[v_m])$ is unstable if and only if there 
exists $I \subseteq J \subseteq \{1,\ldots,m\}$ such that $|I| + |J| = m+1$ and 
$\la v_i, v_j\ra = 0$ for all $i \in I$ and $j \in J$.
\end{proposition}

\begin{proof} Since $(\bbC[\Sym_m]^{(2)})^T$ is generated by determinantal terms, we obtain that a matrix $A = (a_{ij})$ has all determinantal terms equal to zero if and only if it represents an unstable point in $\bbP(\Sym_m)$ with respect to the torus action.  Now the assertion becomes a simple consequence of the Hilbert-Mumford numerical criterion of stability. 

It is obvious that if such subsets $I$ and $J$ exist then all determinantal terms vanish. 
So we are left with proving the existence of the subsets $I$ and $J$ if we have an unstable matrix. 

Let $r:t\mapsto (t^{r_1},\ldots,t^{r_m})$ be a non-trivial one-parameter subgroup of the torus $T$. Permuting the points, we may assume that $r_1\le r_2 \le \ldots \le r_m$. We also have $r_1+\ldots+r_m = 0$.  We claim that there exists $i,j$ such that $i + j = m+1$ and $r_i + r_j \leq 0$.
If not, then each of $r_1 + r_m$, $r_2 + r_{m-1}$, \ldots, 
$r_{\lfloor \frac{m + 1}{2} \rfloor} + r_{\lceil \frac{m + 1}{2} \rceil}$ are strictly positive, 
which contradicts $\sum_i r_i = 0$.  

Since our symmetric matrix $A = (a_{ij})$ is unstable, by the Hilbert-Mumford criterion 
there must exist $r$ such that $\min \{r_i+r_j: a_{ij} \ne 0\} >  0$.  Permute the points, 
if necessary, so that $r_1 \leq r_2 \leq \cdots \leq r_n$.  
Let $i_0$, $j_0$ be such that $r_{i_0} + r_{j_0} \leq 0$ and $i_0 + j_0 = m+1$.
We may assume that $i_0 \leq j_0$ since the above condition is symmetric in $i_0, j_0$.
Now, since the entries of $r$ are increasing, we have that $r_i + r_j \leq 0$ for all $i \leq i_0$ and $j \leq j_0$.  Hence $a_{ij} = 0$ for all $i \leq i_0$ and $j \leq j_0$.  
Now let $I = \{1,\ldots,i_0\}$ and let $J = \{1,\ldots,j_0\}$. 
\end{proof}

Similarly, we can prove the following.

\begin{proposition}\label{P3.4} A point set $([v_1],\ldots,[v_m])$ is semi-stable but not stable if and only if 
$$m = \max \{ |I| + |J| : I \subseteq J \subseteq \{1,\ldots,m\} \mbox{ and } \la v_i, v_j \ra = 0
\mbox{ for all } i \in I, j \in J\}.$$ 
%Do we also need a description of a minimal closed orbit in this case...?
\end{proposition}

\begin{proof}
Suppose that $A = (a_{ij} = \la v_i, v_j \ra)$.   
Let 
$$m'(A) = \max \{ |I| + |J| : I \subseteq J \subseteq \{1,\ldots,m\} \mbox{ and } \la v_i, v_j \ra = 0 
\mbox{ for all } i \in I, j \in J\}.$$ 

Suppose that $A$ is semi-stable but not stable.
Since $A$ is not unstable, we know by the prior Proposition that $m'(A) \leq m$.
So we are left with showing that $m'(A) = m$.

Since $A$ is not stable, there is a one parameter subgroup $r:t\mapsto (t^{r_1},\ldots,t^{r_m})$ such that $r_i + r_j \geq 0$ whenever $a_{ij} \neq 0$. 
We shall re-order the points so that $r_1 \leq r_2 \leq \cdots \leq r_n$.
Recall also that $\sum_i r_i = 0$.    Since some $r_i \neq 0$, we know that 
$r_1 < 0 < r_n$.
We claim there is some $i,j$ such that $i + j = m$ and 
$r_i + r_j < 0$.  Otherwise, each of $r_1 + r_{m-1}$, $r_2 + r_{m-3}$,\ldots,
$r_{\lfloor \frac{m}{2} \rfloor} + r_{\lceil \frac{m}{2}\rceil}$ would be non-negative.
This implies that $r_1 + \cdots + r_{m-1} \geq 0$.  But since $r_n > 0$, 
we have that $r_1 + \cdots r_m > 0$, a contradiction.  Hence, the claim is true.
Now, take $i_0 \leq j_0$ such that $i_0 + j_0 = m$ and $r_{i_0} + r_{j_0} < 0$.  Now, we must have that 
$a_{ij} = 0$ for all $i \leq i_0$ and $j \leq j_0$.   Let $I = \{1,\ldots,i_0\}$ 
and $J = \{1,\ldots,j_0\}$.  Then $I \subseteq J \subseteq \{1,\ldots,m\}$, $|I| + |J| = m$, and 
$a_{ij} = 0$ for all $i \in I$, $j \in J$.  Thus $m'(A) = m$.

Conversely suppose that $m'(A) = m$.  Then $A$ is not unstable 
(if $A$ were unstable, Proposition \ref{P3.3} implies that $m'(A) \geq m+1$).
Let  $I \subseteq J \subseteq \{1,\ldots,m\}$, such that $|I| + |J| = m$,  and 
$a_{ij} = 0$ for all $i \in I$, $j \in J$.  Re-order points if necessary so that 
$I = \{1,\ldots,i_0\}$ and $J = \{1,\ldots,j_0\}$.  
Let $r:t\mapsto (t^{r_1},\ldots,t^{r_m})$ be defined as follows.
Let $r_i = -1$ for $i \leq i_0$, let $r_i = 1$ for $i > m - i_0$, and let $r_i = 0$ otherwise. 
The sum $\sum_i r_i$ is  zero and not all $r_i$ are zero, so this defines a one-parameter subgroup of the torus $T$.
Also, if $r_i + r_j < 0$ and $i \leq j$, then $i \leq i_0$ and $j \leq j_0$, 
which implies that $a_{ij} = 0$.  Thus $r_i + r_j \geq 0$ whenever $a_{ij} \neq 0$.
Hence $A$ is not stable.

\end{proof}

The Second Fundamental Theorem (SFT) of Invariant Theory for the group $\Or(V)$ describes the kernel of the homomorphism $\eqref{pol}$ (see \cite{Procesi}, p. 407, \cite{Weyl}, Chapter 2, \S 17). 

Consider the ideal $\calI(n,m)$ in $\bbC[V^n]^{\Or(V)}$ generated by the \emph{Gram functions}
$$\gamma_{I,J}:(v_1,\ldots,v_m) \mapsto \det \begin{pmatrix}\la v_{i_1}, v_{j_1}\ra&\ldots&\la v_{i_{n+2}},v_{j_{n+2}}\ra\\
\vdots&\vdots&\vdots\\
\la v_{i_{n+2}}, v_{j_1}\ra&\ldots&\la v_{i_{n+2}},v_{j_{n+2}}\ra\end{pmatrix}$$
where $I = (1 < i_1<\cdots < i_{n+2}), J = (1 < j_1<\cdots < j_{n+2})$ are subsets of $[1,m]$. We set
$\gamma_{I} = \gamma_{I,I}$.

The pre-image of this ideal in $\bbC[\Sym^2(V^*)] \cong \bbC[\Sym_m]$ is the determinant ideal $\calD_m(n+1)$ of matrices of rank $\le n+1$.  The SFT asserts that it is the kernel of the homomorphism \eqref{pol}. Then 
$$\Ker(\Phi_T) =\calD_m(n+1) \cap \bbC[\Sym_m]^T$$ 
and it is finitely generated by polynomials of the form $\mathbf{m}\Delta_{I,J}$, where $\mathbf{m}$ is a monomial 
in $X_{ij}$ of degree $(k,k,\ldots,k) - \deg(\Delta_{I,J})$ for some $k \geq 2$.

Our naive hope was that  
$\Ker(\Phi_T)$ is generated only by polynomials of the form
$\mathbf{m}\Delta_{I,J}$, for $\mathbf{m}$ having degree $(2,2,\ldots,2) - \deg(\Delta_{I,J})$.
This is not true even if we restrict it to the open subset of semi-stable points in $\Sym_m$ with respect to the torus action. The symmetric matrix
$$A = \begin{pmatrix}0&0&1&1&1\\
0&0&1&1&1\\
1&1&1&0&0\\
1&1&0&1&0\\
1&1&0&0&1\end{pmatrix}$$
has rank $4$, but for any $(i,j)$ the product $a_{ij}A_{ij}$ (where $A_{ij}$ is the complementary minor) is equal to zero. Thus our naive relations make it appear that $A$ has rank $3$.
Also, $a_{31}a_{42}a_{13}a_{24}a_{55}\ne 0$, so the matrix represents a semi-stable point. It can be shown that no counter-example exists with $m < 5$.

The following $T$-invariant polynomial vanishes on rank $3$ matrices and is nonzero when evaluated on the matrix $A$ above:
$$a_{13}^2a_{14}^2a_{25}^2 \Delta_{\{2,3,4,5\},\{2,3,4,5\}}.$$
Hence we need to consider higher degree relations.
We can at least give a bound on the degree of such relations, again appealing to 
Petersen's factorization theorem.

\begin{proposition}\label{rel_bound}
The ideal $\Ker(\Phi_T)$ is generated by polynomials of the form
$\mathbf{m}\Delta_{I,J}$, for $\mathbf{m}$ having degree at most 
$(2(n+2),2(n+2),\ldots,2(n+2)) - \deg(\Delta_{I,J})$.
\end{proposition}

\begin{proof}
It is clear that  $\Ker(\Phi_T)$ is generated by relations of the form 
$\mathbf{m}\Delta_{I,J}$ where $\mathbf{m}$ is a monomial of degree
$(2k,2k,\ldots,2k) - \deg(\Delta_{I,J})$, for arbitrary $k$.  
Suppose that $k > n+2$.  
The monomial $\mathbf{m}$ corresponds to the multigraph $\Gamma'$ with edges 
$ij$ for each $X_{ij}$ dividing $\mathbf{m}$, counting multiplicity.
Choose any term from $\Delta_{I,J}$; similarly this 
term corresponds to a multigraph $\Gamma''$.  The graph $\Gamma''$ has 
exactly $n+2$ edges.

The union $\Gamma = \Gamma' \sqcup \Gamma''$ is a $2k$-regular graph.  By Petersen's factorization theorem, 
 we know that $\Gamma$ completely factors into $k$ disjoint 
$2$-factors.  Since $k > n+2$, at least one of these $2$-factors is disjoint from $\Gamma''$.  Hence, this $2$-factor must be a factor of $\Gamma'$.
This means that the monomial $\mathbf{m}$ is divisible by the $T$-invariant 
monomial $\mathbf{m}_0$ corresponding to the $2$-factor of $\Gamma'$: 
$$\mathbf{m} = \mathbf{m}_0 \cdot \mathbf{m}'.$$  
Hence, the relation $\mathbf{m}\Delta_{I,J}$ is equal to 
$\mathbf{m}_0 \cdot (\mathbf{m}' \Delta_{I,J})$, where 
$\mathbf{m}' \Delta_{I,J} \in \Ker(\Phi_T)$ has smaller degree.
\end{proof}

\begin{conjecture}\label{Snowden-conjecture}
A recent conjecture of Andrew Snowden (informal communication) implies that there is a bound 
$d_0(n)$ such that $R(n;m)$ is generated in degree $\leq d_0(n)$ for all $m$.   Further, after choosing a minimal set of generators (each of degree $\leq d_0(n)$), his conjecture also implies that there is a bound 
$d_1(n)$ such that the ideal of relations is generated in degree $\leq d_1(n)$ for all $m$.  His conjecture applies to all GIT quotients of 
the form $X^m /\!/ G$ where $G$ is linearly reductive and $X$ is a $G$-polarized projective variety.

One of our goals was to prove (or perhaps disprove) his conjecture for this case of $\bbP(V)^m /\!/ \PO(V)$.  We were not able to do so.  However, we have shown that the second Veronese subring is generated in lowest degree, providing small evidence of the first part of his conjecture.  Furthermore, Proposition \ref{rel_bound} is a small step towards proving an $m$-independent degree bound on the generating set of the ideal (again for the second Veronese subring only).
\end{conjecture}

\section{A toric variety} The variety $\bbS_m = \Sym_m/\!/T = \Proj~\bbC[\Sym_m)]^T$ is a toric variety of dimension $m(m-1)/2$. We identify the character lattice of $\bbC^*{}^m$ with $\bbZ^m$. We have 
$\Sym_m = \oplus \bbC X_{ij}$, where $X_{ij}$ is an eigenvector with the character $\bfe_i+\bfe_j$.
The lattice $M$ of characters  of the torus $(\bbC^*)^{m(m-1)/2}$ acting on $\bbS_m$ is equal to the kernel of the homomorphism $\bbZ^{m(m+1)/2}\to \bbZ^m, \bfe_{ij}\mapsto \bfe_i+\bfe_j.$ It is defined by the matrix $A$ with $(ij)$-spot in a  $k$-th row equal 
\[a_{k,ij} = \begin{cases}
      0& \text{if}\ k \ne i,j,\\
      1& \text{if}\ k = i \ne j, \text{or}\  k = j\ne i,\\
      2& \text{if}\ k = i = j.
\end{cases}\]
Let 
$$S = \{x\in \bbZ_{\ge 0}^{m(m+1)/2}: Ax = 2d(\bfe_1+\ldots+\bfe_m), \text{for some}\ d \ge 0\}$$
be the graded semigroup.
Then 
$$\bbC[\Sym_m]^T = \bbC[S].$$
In other words, the toric variety $\bbS_m$ is equal to the toric space $\bbP_\Delta$, where 
$\Delta_m$ is the convex polytope in $\{x\in \bbR^{m(m+1)/2}:Ax= (2,\ldots,2)\}$ spanned by the vectors $v_\sigma, \sigma\in \frakS_m,$ such that $a_{ij}$ is equal to the number of edges from $i$ to $j$ in the regular graph corresponding to the determinantal term $d_\sigma$. For example, if $m = 3$, $\sigma = (12)$ defines the $v_\sigma$ with $a_{12} = 2, a_{33} = 1$ and $a_{ij} = 0$ otherwise. Thus the number of lattice points in the polytope $\Delta$ is equal to the number $k(m)$ of determinantal terms in a general symmetric matrix. 

\begin{proposition}
$$\#(d\Delta_m)\cap M = \# \{\text{regular graphs with valency}\  2d\}.$$
\end{proposition}

\begin{proof}
This follows easily from Proposition \ref{P3.1}. 
\end{proof}

\section{Examples}
\begin{example}\label{ex5.1} Let $n = 2$ and $m = 3$. We are interested in the moduli space of 3 points in $\bbP^2$ modulo the group of projective transformations leaving invariant a nonsingular conic. The group $\PO(3) \cong \PSL_2$ is a 3-dimensional group. So, we expect a 3-dimensional variety of configurations.

We have five determinantal terms given by the following graphs:

\bigskip
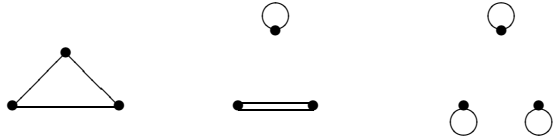
\xy 
(-35,15)*{};(-35,-5)*{};@={(0,0),(14.2,0),(7.1,7.1),(30,0),(40,0),(35,10),(60,0),(70,0),(65,10)}@@{*{\bullet}};
(0,0)*{};(14.2,0)**{}**\dir{-};
(0,0)*{};(7.1,7.1)**{}**\dir{-};
(14.2,0)*{};(7.1,7.1)**{}**\dir{-};
(30,-.5)*{};(40,-.5)**{}**\dir{-};
(30,.5)*{};(40,.5)**{}**\dir{-};
(35,12.1)*\cir<5pt>{};(65,12.1)*\cir<5pt>{};(70,-2)*\cir<5pt>{};(60,-2)*\cir<5pt>{};
\endxy
\bigskip
Let $t_0,t_1,t_2,t_3,t_4$ be generators of the ring $R(2;3)$ corresponding, respectively,  to the triangle, to the three graphs of the second type, and the  one graph of the third type. We have the cubic relation:
$$t_1t_2t_3 - t_0^2t_4 = 0.$$
Thus our variety is a cubic threefold in $\bbP^4$. Its singular locus consists of three lines 
$$t_0 = t_1 = t_2 = 0,\ t_0 = t_1 = t_3 = 0, \ t_0 = t_2 = t_3 = 0.$$
Let $H_i$ be  hyperplane  section of the cubic by the coordinate hyperplane $t_i = 0$. Then 
\begin{itemize}
\item 
$H_0$: point sets with two points  conjugate with respect to the fundamental conic.  $H_0$ is the union of three planes $\Lambda_i:t_0 = t_i = 0, i = 1,2,3$.
\item $H_4$: one of the points lies on the fundamental  conic. It is the union of three planes $\Pi_i:t_4 = t_i = 0, i = 1,2,3$.
\item $H_i$ is the union of two planes $\Lambda_i$ and $\Pi_i, i = 1,2,3$.
\item  $\Lambda_i\cap \Lambda_j$ is a singular line on $\bbS_3$, the locus of point sets where one point is  the intersection point of the polar lines of two other points;
\item $\Pi_i\cap \Pi_j$: two points are on the fundamental conic;
\item $\Pi_i\cap \Lambda_i$: two points are conjugate, the third point is on the conic;
\item $\Pi_i\cap \Lambda_j, i \ne j$: one point is on the conic, and another point lies on the tangent to the conic at this point;
 \item $\Lambda_1\cap \Lambda_2\cap \Lambda_3$ is the point representing the orbit of ordered self-conjugate triangles;
\item $\Pi_1\cap \Pi_2\cap \Pi_3$ is the point representing the orbit of ordered sets of points on the  fundamental conic.
\end{itemize}

The singular point $\Lambda_1\cap \Lambda_2\cap \Lambda_3 = [0,0,0,0,1]$ represents the  orbit of ordered self-polar  triangles.  Recall that unordered self-conjugate triangles are parameterized by the homogeneous space $\PO(3)/\frakS_{4}$. It admits a smooth compactification isomorphic to the Fano threefold of degree 5 and index 2 \cite{Mukai}, Theorem (2.1) and Lemma (3.3) (see also \cite{CAG}, 2.1.3).

\end{example}
\begin{example} Let us look at the variety $\bfO_1^3$. It isomorphic to the subvariety of $\bfO_2^3$ representing collinear triples of points.
The equation of the determinant of the Gram matrix of three points is 
\beq\label{hyp}
2t_0-t_1-t_2-t_3+t_4 = 0.
\eeq
It a hyperplane section of $\bbS_3$ isomorphic to a cubic surface $S$ in $\bbP^3$ with equation 
\beq\label{segre}
t_1t_2t_3 + 2t_0^3-t_0^2t_1-t_0^2t_2-t_0^2t_3 = 0.
\eeq
The surface is projectively isomorphic to the 4-nodal Cayley cubic surface given by the equation
$$x_0x_1x_2+x_0x_2x_3+x_0x_1x_3+x_1x_2x_3 = 0.$$
 Its singular points are $[t_0,t_1,t_2,t_3] = [1,1,1,1], 
[0,0,0,1], [0,0,1,0], [0,1,0,0]$. Since the surface is irreducible, and all collinear sets of points satisfy \eqref{hyp}, we obtain that the surface represents the locus of collinear point sets. It is also isomorphic to the variety $\bfO_1^3$ of 3 points on $\bbP^1$. The additional singular point $[1,1,1,1]$ not inherited from the singular locus of $\bfO_2^3$ is the orbit of three  collinear points $[v_1], [v_2], [v_3]$ such that the determinantal  terms of the Gram matrix $G(v_1,v_2,v_3)$ are all equal.  This is equivalent to that all principal minors are equal to zero and the squares of the discriminant terms $d_{(123)}$ and $d_{(321)} $ are equal. This gives two possible  points $[t_0,t_1,t_2,t_3] = [\pm 1,1,1,1]$. We check that the point $[-1,1,1,1]$  does not satisfy \eqref{segre}. Thus the point $[1,1,1,1]$ is determined by the condition that the principal minors of the Gram matrix $G(v_1,v_2,v_3)$ are equal to zero. This implies that $[v_1] = [v_2] = [v_3]$. It follows from the stability criterion that this point is not one of the two isotropic points. 

It is immediate that $R(1;2)^{(2)}$ is freely generated by  two determinantal terms and hence $\bfO_1^2\cong \bbP^1$.  The three projections $\bfO_1^3$ to $\bfO_1^2$ are  regular  map.  If we  realize $\calS_3$ as the image of the anti-canonical system of the blow-up of 6 vertices of a complete quadrilateral  in the plane, then the three maps are  defined by the linear system of conics through three subsets of 4 vertices no three lying on one side of the quadrilateral. One can show that these are the only regular maps from $\calS_3$ to $\bbP^1$.

\bigskip
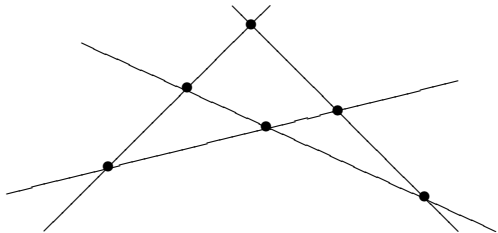
\xy 
(-35,0)*{};@={(27.5,27.5),(8.5,8.5),(29.5,13.8),(39,16), (19,19),(50.5,4.5)}@@{*{\bullet}};
(30,30)*{};(0,0)**{}**\dir{-};
(25,30)*{};(55,0)**{}**\dir{-};
(-5,5)*{};(55,20)**{}**\dir{-};
(60,0)*{};(5,25)**{}**\dir{-};

\endxy
\bigskip

Finally, observe, that we can use the conic to identify the plane with its dual plane. In this interpretation a triple of points becomes a triple of lines, the polar lines of the points with respect to the conic. Intersecting each line with the conic we obtain three ordered pairs of points on a conic. 

Note that a set of six distinct points on a nonsingular conic can be viewed as the set of Weierstrass points of a hyperelliptic curve $C$ of genus 2. An order on this set defines a symplectic basis of the $\bbF_2$-symplectic space $\Jac(C)[2]$ of 2-torsion points of its Jacobian variety $\Jac(C)$. The GIT-quotient of the subvariety of $(\bbP^2)^6$ of ordered points on a conic by the group $\SL_3$ is isomorphic to the Igusa quartic in $\bbP^4$ (see \cite{DO}, Chapter 1, Example 3). A partition of the set of Weierstrass points in three pairs defines a maximal isotropic subspace in $\Jac(C)[2]$. An order of the three pairs chooses a basis in this space. The moduli space of principally polarized abelian surfaces $A$ equipped with a symplectic basis 
in  $A[2]$ is isomorphic to the quotient of the Siegel space $\calZ_2 = \{X\in \Sym_4:\textup{Im}(X) > 0\}$ by the group 
$\Gamma(2) = \{M\in \Sp(4,\bbZ):A\equiv I_4 \mod 2\}$. The moduli space of principally polarized abelian surfaces together with a choice of a basis in a maximal isotropic subspace of 2-torsion points is isomorphic to the quotient of $\calZ_2$ by the group $\Gamma_1(2) = \{M = \begin{pmatrix}A&B\\
C&D\end{pmatrix}:A-I_2\equiv C\equiv 0 \mod 2\}$. Thus, we obtain that our variety $\bfO_2^3$ is naturally birationally isomorphic to the quotient $\calZ_2/\Gamma_1(2)$ and this variety is isomorphic to the quotient of $\calZ_2/\Gamma(2)$ by the group $G = \frakS_2\times \frakS_2\times \frakS_2$. The Satake compactification of $\calZ_2/\Gamma(2)$ is isomorphic to the Igusa quartic. In \cite{Mukai2}, S. Mukai shows that the Satake compactification of $\calZ_2/\Gamma_1(2)$ is isomorphic to the double cover of $\bbP^3$ branched along the union of four coordinate hyperplanes. It is easy to see that it is birationally isomorphic to the cubic hypersurface defining $\bfO_2^3$. A remarkable result of Mukai is that the Satake compactifications of $\calZ_2/\Gamma(2)$ and $\calZ_2/\Gamma_1(2)$ are isomorphic. 
\end{example}

\begin{remark} Assume $m = n+1$. Fix a volume form on $V$ and use it to identify the linear spaces $V^*$ and $\bigwedge^nV$. This identification is equivariant with respect to the action of $\Or(V)$ on $V$ and $\Or(V^*)$, where the orthogonal group of $V^*$ is with respect to the dual quadratic form on $V^*$. Passing to the 
configuration spaces, we obtain a natural birational involution $F:\bfO_n^{n+1}\dasharrow \bfO_n^{n+1}$. If $G$ is the Gram matrix of vectors $v_1,\cdots,v_{n+1}$, then the Gram matrix $G^*$ of the vectors $w_i = v_1\wedge\cdots\wedge v_{i-1}\wedge v_{i+1}\wedge\cdots\wedge v_{n+1}\in V^*$ is equal to the adjugate matrix of $G$ (see \cite{CAG}, Lemma 10.3.2). In the case $n = 2$, the birational involution correspond to the involution defined by conjugate triangles (see loc.cit., 2.1.4). Using the modular interpretation of $\bfO_2^3$ from the previous example, the involution $F$ corresponds to the Fricke (or Richelot) involution of $\calZ_2/\Gamma_1(2)$ (see  \cite{Mukai2}, Theorem 2).
\end{remark} 

\begin{example} Now let us consider the variety $O_1^4$ of 4 points in $\bbP^1$ modulo $\PO(2) \cong \bbC^*\times \bbZ/2\bbZ$. It is another threefold.  First we get the 5-dimensional toric variety of symmetric matrices of size 4. The coordinate ring is generated by 17 (3+4+6+3+1) determinantal terms:

\bigskip
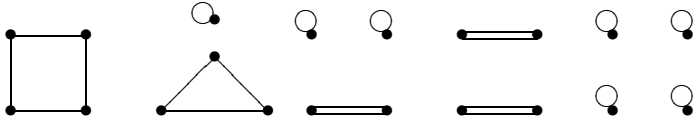
\xy
(-15,0)*{};@={(0,0),(10,0),(10,10),(0,10),(20,0),(34.2,0),(27.1,7.1),(40,0),(40,10),(50,10),(50,0),(27.1,12),(60,0),(60,10),(70,0),(70,10),(80,0),(80,10),(90,0),(90,10)}@@{*{\bullet}};
%square
(0,0)*{};(10,0)**{}**\dir{-};
(10,10)*{};(10,0)**{}**\dir{-};
(10,10)*{};(0,10)**{}**\dir{-};
(0,10)*{};(0,0)**{}**\dir{-};
%triangle
(20,0)*{};(34.2,0)**{}**\dir{-};
(20,0)*{};(27.1,7.1)**{}**\dir{-};
(27.1,7.1)*{};(34.2,0)**{}**\dir{-};
(40,.5)*{};(50,.5)**{}**\dir{-};
(40,-.5)*{};(50,-.5)**{}**\dir{-};
(60,.5)*{};(70,.5)**{}**\dir{-};
(60,-.5)*{};(70,-.5)**{}**\dir{-};
(60,10.5)*{};(70,10.5)**{}**\dir{-};
(60,9.5)*{};(70,9.5)**{}**\dir{-};
(25.5,13)*\cir<4pt>{};(39.2,12)*\cir<4pt>{};(49.2,12)*\cir<4pt>{};
(79.2,12)*\cir<4pt>{};(89.2,12)*\cir<4pt>{};(79.2,2)*\cir<4pt>{};(89.2,2)*\cir<4pt>{};
\endxy

\bigskip
Let $x_1,x_2,x_3, y_1,y_2,y_3,y_4, z_1,\ldots,z_6, u_1,u_2,u_3,v$ be the variables. We have additional equations expressing the condition that the rank of matrices is less than or equal to 2. One can show that the equations are all linear:
$$ \ a_{ij}\det A_{ij} = 0.$$
Their number is equal to $10$ but there are 3 linear dependencies found by expanding the determinant expression along columns. 
\end{example}

We may also consider the spaces $\bfS\bfO_n^m = \bbP(V)^m/\!/\Or^+(V)$, where $\Or^+(V) = \Or(V)\cap \SL(V)$ is the special orthogonal group. Note that $\PO(V) \cong \PO^+(V)$ if $\dim V$ is odd. Thus we will be interested only in the case when $\dim V$ is even. In this case $\PO^+(V)$ is a subgroup of index 2 in $\PO(V)$, so the variety $\bfS\bfO_n^m$ is a double cover of $\bfO_n^m$.  We have 
$$\bfS\bfO_n^m = \Proj\ R^+(n;m),$$
where $R^+(n;m) = \oplus_{d=0}^\infty (S^d(V)^*{}^{\otimes m})^{\Or^+(V)}$. There are more invariants now. The additional invariants in $\bbC[V^m]^{\Or^+(V)}$ are the Pl\"ucker brackets 
$$p_{i_1,\ldots,i_{n+1}}:(v_1,\ldots,v_m) \mapsto v_{i_1}\wedge \ldots\wedge v_{i_{n+1}},$$
where we have fixed a volume form on $V$. 
There are additional basic relations (see \cite{Weyl}, Chapter 2, \S 17)
\begin{eqnarray}\label{rels}
p_{i_1,\ldots,i_{n+1}}p_{j_1,\ldots,j_{n+1}} - \det ( [i_\alpha,i_\beta])_{1\le \alpha,\beta\le n+1} = 0,\\
\sum_{j=1}^{n+1} (-1)^j\sum p_{i_1,\ldots,\hat{i_j},\ldots, i_{n+1}}[i_{j},i_{n+2}]  = 0.
\end{eqnarray}
The graded part $R^+(n;m)_{d}$ is spanned by the monomials $p_{I_1}\cdots p_{I_k}[i_1j_1]\cdots [i_s,j_s]$ where each index $j\in [1,m]$ appears exactly $d$ times. Using the first relation in \eqref{rels}, we may assume that at most  one Pl\"ucker coordinate $p_I$ appears.  Also we see that the product of any two elements in $R^+(n;m)_d$ belongs to $R(N;m)_{2d}$.

\section{Points in $\bbP^1$ and generalized Cayley cubics}

The group $\SO(2)$ is isomorphic to the one-dimensional complex torus $\bbC^*$. We choose projective coordinates in $\bbP^1$ to identify a quadric in $\bbP^1$ with the set $Q = \{0,\infty\}$ so that $\SO(2)$ acts by $\lambda: [t_0,t_1]  \mapsto [\lambda t_0,\lambda^{-1}t_1]$. The points on $Q$ are the fixed points of $\SO(2)$. The group $\Or(2)$ is generated by  $\SO(2)$ and the transformation $[t_0,t_1]  \mapsto [t_1,t_0]$.

Recall that there is a Chow quotient $(\bbP^1)^m/\!/\bbC^*$ defined by the quotient fan of the toric variety  $(\bbP^1)^m$ (see \cite{KSZ}). 

\begin{lemma} Consider $(\bbP^1)^m$ as a toric variety, the Cartesian product of the toric varieties $\bbP^1$. Then Chow quotient $(\bbP^1)^m/\!/\SO(2)$ is isomorphic to the toric variety $X(A_{m-1})$ associated to the root system of type $A_{m-1}$ defined by the fan in the dual  lattice of the root lattice of type $A_{m-1}$ formed by the Weyl chambers.
\end{lemma}

\begin{proof} The toric variety $(\bbP^1)^m$ is defined by the complete fan $\Sigma$ in the lattice $\bbZ^{m}$ with 1-skeleton formed by the rays $\bbR_{\ge 0}e_i$ and $\bbR_{\le 0}e_i, \ i = 1,\ldots,m$.  The action of $\SO(2)$ on the torus $(\bbC^*)^m\subset (\bbP^1)^m$ is defined by the surjection of the lattices $\bbZ^m\to \bbZ$ given by the map $e_i\to e_1, i = 1,\ldots,m$. Thus the  lattice $M$ of characters of the torus $(\bbC^*)^{m-1}$ acting on $(\bbC^*)^m/\bbC^*$ can be identified with the sublattice of $\bbZ^m$ spanned by the vectors $\alpha_1 = e_1-e_2,\ldots,\alpha_{m-1} = e_{m-1}-e_m$. This is the root lattice of type $A_{m-1}$. The dual lattice $N$ is the  lattice $\bbZ^{m}/\bbZ e$, where $e = e_1+\ldots+e_{m}$. 
The quotient fan is defined as follows. For any coset $y \in N_\bbR = \bbR^m/\bbR e$, one considers the set
$$\calN(y) = \{\sigma\in \Sigma: y+\bbR e \cap \sigma \ne \emptyset\}.$$
A coset  $y+\bbR e\in N_\bbR$ is called admissible if $\calN(\phi)\ne \emptyset$.  Two admissible cosets  $y+\bbR e$ and $y'+\bbR e$ are called equivalent if $\calN(y) = \calN(y')$. The closure if each equivalence class of admissible cosets is a rational polyhedral convex cone in $N_\bbR$ and the set of such cones  defines a fan  $\Sigma'$ in $N_\bbR$ which is the quotient fan. 

In our case, $\Sigma$ consists of open faces  $\sigma_I$ of the $2^m$ $m$-dimensional cones  
$$\sigma_{I,J} = \{(x_1,\ldots,x_m)\in \bbR^m: (-1)^{\delta_I}x_i \ge 0, (-1)^{\delta_J}x_i \le 0\},$$
where $I,J$ are subsets of $[1,m]$ such that  $I\cup J = [1,m]$ and $\delta_K$ is a delta-function of a subset $K$ of $[1,m]$. The  cones  of maximal dimension correspond to  pairs of complementary subsets $I,J$. The $k$-dimensional cones correspond to the pairs $I,J$ with $\#I\cap J = m-k$.

 Let $y = (y_1,\ldots,y_m)\in \bbR^m$  with $y_i \ne y_j, i \ne j$  and let $s\in \frakS_m$ be a   unique permutation such that $y_{s(1)}> y_{s(2)}> \ldots > y_{\sigma(m)}$. Then $y+\bbR e$ intersects $\sigma_{I,J}$ if and only if $s(I) = \{1,\ldots,k\}$ for some $k \le m$ or $\emptyset$ and  $J = [1,m]\setminus I$. Since $\frakS_m$ has only one orbit on the set of  pairs of complementary subsets of $[1,m]$, we see that the interiors of maximal cones in the quotient fan are obtained from the image of the subset
 $$\{y \in \bbR^m: y\cdot (e_i-e_{i+1}) \ge 0, i = 1,\ldots,m-1\}$$
 in $N_\bbR$. This is exactly one of the Weyl chambers in $N_\bbR$. All other cones in the quotient fan are translates of the faces of the closure of this chamber. This proves the assertion.
 \end{proof}
  
It is known that the toric variety $X(A_{m-1})$ is isomorphic to the blow-up of $\bbP^{m-1}$ of the faces of the coordinate simplex (see, for example, \cite{DL}, Lemma 5.1). Let 
$$\tau_{m-1}:\bbP^{m-1}\dasharrow \bbP^{m-1}, \ [t_0,\ldots,t_{m-1}]\mapsto [1/t_0,\ldots,1/t_{m-1}]$$
be the standard Cremona transformation of $\bbP^{m-1}$. The variety $X(A_{m-1})$ is isomorphic to a minimal resolution of indeterminacy points of %\tau_{m-1}$ (see \cite{CAG}, Example 7.2.5). Equivalently, $X(A_{m-1})$ is isomorphic to the closure of the graph of $\tau_{m-1}$ in $\bbP^{m-1}\times \bbP^{m-1}$. It is given by the $2\times 2$-minors of the matrix
$$\begin{pmatrix}F_0(x)&F_1(x)&\ldots&F_{m-1}(x)\\
y_0&y_1&\ldots&y_{m-1}\end{pmatrix},$$
where  $F_i(x) = (x_0\cdots x_{m-1})/x_i$. It follows from this formula, that the standard involution $\tau_{m-1}$  of $X(A_{m-1})$ is induced by the switching involution $\iota$ of the factors of $\bbP^{m-1}\times \bbP^{m-1}$. The image of composition of the embedding $X(A_{m-1})$ in $\bbP^{m-1}\times \bbP^{m-1}$ and the Segre embedding $\bbP^{m-1}\times \bbP^{m-1}\hookrightarrow \bbP^{m^2-1}$ is equal to the intersection of the Segre variety with the linear subspace of dimension $m-1$ defined by 
\beq\label{eqq}
t_{00} = t_{11} =\ldots = t_{m-1m-1},
\eeq
where we use the coordinates  $t_{ij} = x_iy_j$ 
 in $\bbP^{m^2-1}$. So $X(A_{m-1})$ is isomorphic to a closed smooth subvariety of $\bbP^{m(m-1)}$ of degree $\binom{2(m-1)}{m-1}$.

Consider the embedding of $(\bbP^{m-1}\times \bbP^{m-1})/\la \iota \ra$ in $\bbP^{\tfrac{1}{2}(m+2)(m-1)}$ given by the linear system of symmetric divisors of type $(1,1)$. Its image is equal to the secant variety of the Veronese variety $v_2(\bbP^{m-1})$ isomorphic to the symmetric square $\Sym^2\bbP^{m-1}$ of $\bbP^{m-1}$. The image of $X(A_{m-1})/\la \tau_{m-1}\ra$ in $\bbP^{\half (m+2)(m-1)}$ is equal to the intersection of the secant variety  with a linear subspace $L$ of codimension $m-1$ given by \eqref{eqq}. It is known that the singular locus of the secant variety is equal to the Veronese variety. The singular locus of the embedded $X(A_{m-1})/\la \tau_{m-1}\ra$ is equal to the intersection of $L$ with the Veronese subvariety $v_2(\bbP^{m-1})$ and consists of $2^{m-1}$ points. We have $\dim L = \tfrac{1}{2}m(m-1)$ and $\deg  \Sym^2\bbP^{m-1}) = \half \binom{2(m-1)}{m-1}$. So  $X(A_{m-1})/\la \tau_{m-1}\ra$ embeds into 
$\bbP^{\half m(m-1)}$ as  a subvariety of degree $\half \binom{2(m-1)}{m-1}$ with $2^{m-1}$ singular points locally isomorphic to the singular point of the cone over the Veronese variety $v_2(\bbP^{m-2})$. We call it the \emph{generalized Cayley cubic} and denote it by $\Cay_{m-1}$. 

It follows from above that $\Cay_{m-1}$ is isomorphic to the subvariety of the projective space of symmetric $m\times m$-matrices with the conditions that the rank is equal to 2 and the diagonal elements are equal.

In the case when $m = 3$, the variety $X(A_2)$ is a Del Pezzo surface of degree 6, the blow up of 3 non-collinear points in $\bbP^2$ and $\Cay_2$ is isomorphic to the Cayley 4-nodal cubic surface in $\bbP^3$. The variety  $\Cay_3$  is a 3-dimensional subvariety of $\bbP^6$ of degree 10 with 8 singular points locally isomorphic to the cone over the Veronese surface. 

It is known that the Chow quotient  birationally dominates all the GIT-quotients \cite{Kapranov2}, Theorem (0.4.3). 
So we have a $\frakS_m$-equivariant birational morphism
$$\Phi_m:X(A_{m-1}) \to \bfS\bfO_1^m.$$
which, after dividing by the involution $\tau_{m-1}$, defines a $\frakS_m$-equivariant birational morphism
$$\Phi_m^c:\Cay_{m-1}\to \bfO_1^m.$$
For example, take $m = 3$. The variety  $\Cay_2$ is the Cayley 4-nodal cubic, the morphism $\Phi_m$ is an isomorphism.  Take $m = 4$. We know from Proposition \ref{P3.4} that the variety 
$\bfS\bfO_1^4$ has 6 singular points corresponding to strictly semi-stable points defined by vanishing of two complementary principal matrices of the Gram matrix. They are represented by the point sets of the form $(a,a,b,b)$, where $a, b \in \{0,\infty\}$.  The morphism $\Phi_4$ resolves these points with the exceptional divisors equal to the exceptional divisors of $X(A_{3}) \to\bbP^{3}$ over the edges of the coordinate tetrahedron.  The  morphism $\Phi_4^c$ resolves 3 singular points of $\bfO_4^1$ and leaves unresolved the 8 singular points coming from the fixed points of $\tau_3$. Altogether, the variety $\bfO_4^1$ has 11 singular points: 8 points locally isomorphic to the cones of the Veronese  surface and 3 conical double points. The latter three singular points correspond to strictly semi-stable orbits.

It is known  that $X(A_{m-1})$ is isomorphic to the closure of a general maximal torus orbit in $\PGL_m/B$, where $B$ is a Borel subgroup \cite{Klyachko}, Theorem 1. Let $P$ be a parabolic subgroup containing $B$ defined by a subset  $S$ of the set of simple roots, and $W_S$ be the subgroup of the Weil group $\frakS_m$ generated by simple roots in $S$. Let $\phi_S: \PGL_m/B ]\to \PGL_m/P$ be the natural projection. The image of $X(A_{m-1})$ in $\PGL_m/P$ is a toric variety $X(A_{m-1})_S$ defined by the fan whose maximal cones are $\frakS_m$-translates of the cone $W_S\sigma$, where $\sigma$ a fundamental chamber (\cite{FH}, Theorem 1). The morphism $\phi_S:X(A_{m-1})\to X(A_{m-1})_S$ is a birational morphism which is easy to describe.

We believe, but could not find a proof, that for odd $m$, the morphism $\Phi_m$ and $\Phi_m^c$ are isomorphisms. If $m$ is even, then the morphism $\Phi_m$ is equal to the morphism $\phi_S$, where $S$ is the complement of the central vertex of the Dynkin diagram of type $A_{m-1}$.

\section{Rational functions}

First we shall prove the rationality of our moduli spaces.

\begin{theorem}  The varieties $\bfO_n^m$ are rational varieties. 
\end{theorem}
\begin{proof}  The assertion is trivial when $m = n+1$ because in this case the variety is isomorphic to the toric variety $\bbS_m$. If  $m < n+1$, a general point set spans $\bbP(W)$, where $W$ is a subspace of $\bbP(V)$ of dimension $m$. Since $\Or(V)$ acts transitively on a dense orbit of the Grassmannian $G(m,V)$ (the subspaces containing an orthogonal basis), we may transform a general set to a subset of a fixed $\bbP(W)$. This shows that the varieties $\bfO_n^m$ and  $\bfO_{m-1}^m$ are birationally isomorphic. If $m > n+1$, we use  the projection map $\bfO_n^m\dasharrow \bfO_n^{m-1}$ onto the first $m-1$ factors.   It is a rational map with general fibre isomorphic to $\bbP^n$. Its geometric generic fibre is isomorphic to the projective space over the algebraic closure of the field $K$ of rational functions of $\bfO_n^{m-1}$. In other words, the generic fibre is a Severi-Brauer variety over $K$ (see \cite{Serre}, Ch. X, \S 6). The rational map has a rational section $(x_1,\ldots,x_{m-1}) \mapsto (x_1,\ldots,x_{m-1},x_{m-1})$. Thus the generic fibre is a Severi-Brauer variety with a rational point, hence isomorphic to the projective space over $K$ (loc. cit., Exercise 1). Thus the field of rational functions on $\bfO_n^m$ is a purely transcendental extension of $K$, and by induction on $m$, we obtain that $\bfO_n^m$ is rational. 
\end{proof}

We know that the ring $R(n;m)^{(2)}$ is generated by  determinantal terms $d_\sigma$ of the Gram matrix of $m$ points. If we take $\sigma$ to be a transposition $(ab)$, then the ratio $d_\sigma/d_{(12\ldots m)}$ is equal to 
\beq\label{rat1_0}
R_{ab} = [ab]^2 /[aa][bb].
\eeq
More generally for any cyclic permutation $\sigma = (a_1\ldots a_k)$ we can do the same to obtain the rational invariant function
\beq\label{rat1_1}
R_{a_1\ldots a_k} = \frac{[a_1a_2] \cdots [a_{k-1}a_k] [a_{k-1}a_k][a_ka_1]}{[a_1a_1]\cdots [a_ka_k]}.
\eeq
Writing any permutation as a product of cycles, we see that the field of rational functions on $\bfO_n^m$ is generated by functions $R_{a_1\ldots a_k}$. Note that 
$$ R_{a_1\ldots a_k}^2 = R_{a_1a_2}\cdots R_{a_{k}a_1}.$$
We do not know whether a transcendental basis of the field can be chosen among the functions $R_{a_1\ldots a_k}$ or their ratios.

\section{Complex spheres}
A $(n-1)$-dimensional sphere is given by equation  in $\bbR^{n}$ of the form
$$\sum_{i=1}^{n}(x_i-a_i)^2 = R^2.$$
After homogenizing, we get the equation in $\bbP^{n}(\bbR)$
\beq\label{sphere1}
\textsf{Q}:\sum_{i=1}^n(x_i-a_ix_0)^2-R^2x_0^2 = 0.
\eeq
The hyperplane section $x_0 = 0$ is a sphere in $\bbP^{n-1}(\bbR)$ with equation
\beq\label{sphere2}
\textsf{Q}_0:\sum_{i=1}^nx_i^2 = 0.
\eeq
The quadric has no real points, and for this reason, it is called the imaginary sphere. Now we abandon the real space and replace $\bbR$ with $\bbC$. Equation \eqref{sphere1} defines  a \emph{complex sphere}. A coordinate-free definition of a complex sphere is a nonsingular quadric hypersurface $Q$ in $\bbP^{n}$ intersecting a fixed hyperplane $H_0$ along a fixed nonsingular quadric $\textsf{Q}_0$ in $H_0$. In the real case, we additionally assume that $\textsf{Q}_0(\bbR) =\emptyset$. If we choose coordinates such that $\calQ_0$ is given by equation \eqref{sphere2},  then a quadric in $\bbP^{n+1}(\bbC)$ containing the imaginary sphere has an equation
$$b(\sum_{i=1}^nx_i^2) -2x_0(\sum_{i=0}^na_ix_i)= 0.$$
If $b\ne 0$,  we may assume that $b = 1$ and  rewrite the equation in the form
$$\sum_{i=1}^n(x_i-a_ix_0)^2-(2a_0+\sum_{i=1}^na_i^2)x_0^2 = 0,$$
so it is a complex sphere. 
Consider the rational map given by the linear system of quadrics in $\bbP^n$ containing the fixed quadric $\textsf{Q}_0$ with equation \eqref{sphere2}. We can 
choose a basis formed  by the quadric 
$\textsf{Q}_0$ and the quadrics $V(x_0x_i), i = 0,\ldots,n$.  This defines  a rational map 
$\bbP^n\dasharrow \bbP^{n+1}$ given by formulae
$$[x_0,\ldots,x_n] \mapsto [t_0,\ldots,t_{n+1}] = [-2x_0^2,-2x_0x_1,\ldots,-2x_0x_n,\sum_{i=1}^n x_i^2].$$
The  image of this map is a nonsingular quadric  in $\bbP^{n+1}$ given by the equation
$\calQ = V(q)$, where
\beq\label{fq}
q =  2t_0t_{n+1}+\sum_{i=1}^{n}t_i^2 = 0.
\eeq
We call $\calQ$   the \emph{fundamental quadric}. The quadratic form $q$  defines  a symmetric bilinear form on $V$ whose value on vectors $v,w\in V$ are denoted by $\la v,w\ra$. The pre-image of a hyperplane section $\sum A_it_i = 0$ is a complex sphere, or its degeneration. For example,  the sphere corresponding to a hyperplane which is tangent to the quadric has the zero radius, and hence, it is  defined by a singular quadric. 

 The idea of replacing a quadratic equation of a sphere by a linear equation goes back to Moebius and Chasles in 1850, but was developed by Klein and Lie twenty years later. The spherical geometry, as it is understood in Klein's  Erlangen Program becomes isomorphic to the orthogonal geometry. More precisely, the \emph{Inversive group} of birational transformations of $\bbP^n$ sending spheres to spheres or their degenerations is isomorphic to the projective orthogonal group $\PO(n+2)$. 

Let us use the quadric $\calQ$ to define a \emph{polarity duality} between points and hyperplanes  in $\bbP^{n+1}$.   If we use the equation of $\calQ$ to define a symmetric bilinear form in $\bbC^{n+2}$, the polarity is just the orthogonality of lines and hyperplanes with respect to this form. 
Under the polarity, hyperplanes become points, and hence spheres in $\bbP^n$ can be identified with points in $\bbP^{n+1}$.   

Explicitly, a point $\alpha = [\alpha_0,\ldots,\alpha_{n+1}]\in \bbP^{n+1}$ defines the sphere

\beq\label{sphere3}
S(\alpha): \alpha_0\sum_{i=1}^nx_i^2-2\sum_{i=1}^{n}\alpha_ix_0x_i-2\alpha_{n+1}x_0^2 = 0.
\eeq
By definition, its \emph{center} is the point $c = [\alpha_0,\alpha_1,\ldots,\alpha_n]$, its \emph{radius square} $R^2$ is defined by the formula
\beq\label{radius1}
\alpha_0^2R^2 = \sum_{i=1}^n \alpha_i^2+2\alpha_0\alpha_{n+1} = q(\alpha_0,\ldots,\alpha_{n+1}). 
\eeq

Computing the discriminant $D$ of the quadratic form in \eqref{sphere3}, we find
\beq\label{discr}
D = \alpha_{0}^{n-1}(2\alpha_0\alpha_{n+1}+\sum_{i=1}^n \alpha_i^2) =  \alpha_{0}^{n-1}q.
\eeq
This proves the following.

\begin{proposition} A complex sphere $S(\alpha)$ is singular if and only if $a_0R^2 = 0$, or,equivalently, the point  $\alpha \in\bbP^{n+1}$ lies on the fundamental quadric $\calQ$. If $a_0 = 0$, it contains the hyperplane at infinity. If $a_0 \ne 0$ and $R^2 = 0$, the center  is its unique singular point.
\end{proposition}

\begin{remark} Spheres of radius zero are points on the fundamental quadric. Thus the spaces $\bfO_n^m$ contain as its closed subsets the moduli space of $m$ points on the fundamental quadric modulo the automorphism group of the quadric. For example, when $n = 2$, this is the moduli space 
$P_1^m = (\bbP^1)^m/\!/\SL(2)$ intensively studied in many papers (see, for example, \cite{DO}, \cite{Howard}). 

\end{remark}

Many geometrical mutual properties of complex spheres are expressed by vanishing of some orthogonal invariant of point sets in $\bbP^{n+1}$. We give here only some simple examples. 

We define two complex spheres in $\bbP^n$ to be \emph{orthogonal} to each other if the corresponding points in $\bbP^{n+1}$ are \emph{conjugate} in the sense that one point lies on the polar hyperplane to another point.

\begin{proposition} Two real spheres in $\bbR^{n}$ are orthogonal to each other (i.e. the radius-vectors at their intersection points are orthogonal) if and only if the corresponding complex spheres are orthogonal in the sense of the previous definition.
\end{proposition}

\begin{proof} Let 
$$\sum_{i=1}^n(x_i-a_i)^2 = r^2,\quad \sum_{i=1}^n(x_i-b_i)^2 = r'{}^2,$$
be two orthogonal spheres. Let $(c_1,\ldots,c_n)$ be their intersection point. Then we have 
$$0 = \sum_{i=1}^n (c_i-a_i)(c_i-b_i) =  \sum_{i=1}^nc_i^2-\sum_{i=1}^n (a_i+b_i)c_i+ \sum_{i=1}^na_ib_i,$$
$$\sum_{i=1}^n(c_i-a_i)^2 = r^2,\quad \sum_{i=1}^n(c_i-b_i)^2 = r'{}^2.$$
This gives the equality
$$2\sum_{i=1}^na_ib_i-\sum_{i=1}^na_i^2-\sum_{i=1}^nb_i^2+r^2+r'{}^2 = 0.$$
It gives  is a necessary and a sufficient condition that  two spheres intersect orthogonally. It is clear that the condition does not depend on a choice of an intersection point. The corresponding complex spheres correspond
to points $[1,a_1,\ldots,a_n, \frac{1}{2}(r^2-\sum a_i^2)]$ and $[1,b_1,\ldots,b_n, \frac{1}{2}(r^2-\sum b_i^2)]$. The condition that two points $[\alpha_0,\ldots,\alpha_{n+1}]$ and $[\alpha_0,\ldots,\alpha_{n+1}]$ are conjugate is 
$$\alpha_0\beta_{n+1}+\alpha_{n+1}\beta_0+\sum_{i=1}^n \alpha_i\beta_i = 0.$$
So we see that the two conditions  agree.
\end{proof} 

 For convenience of notation, we denote  $x = \bbC v\in \bbP^n$ by $[v]$. We use the symmetric form $\la v,w\ra$ in $V$ defined by the fundamental quadric. 

We have learnt the statements of the following two propositions from \cite{Krasner}.

\begin{proposition}\label{P2.3}Two complex spheres $S([v])$ and $S([w])$ are tangent at some point if and only if 
$$\det \begin{pmatrix} \la v,v\ra&\la v,w\ra\\
\la w,v\ra&\la w,w\ra\end{pmatrix} = 0.
 $$
\end{proposition}

\begin{proof} Let $\lambda \phi_1 +\mu \phi_2$ be a one-dimensional space of quadratic forms in $V$ and $V(\lambda \phi_1 +\mu \phi_2)$ be the corresponding  pencil of quadrics in $\bbP^n$. We assume that it contains   a nonsingular quadric. Then the equation $\discr(\lambda \phi_1 +\mu \phi_2) = 0$ is a homogeneous form of degree $n+1$ whose zeros define singular quadrics in the pencil. The quadrics $V(\phi_1)$ and $V(\phi_2)$ are tangent at some point $p$ if and only if $p$ is a singular point of some  member of the pencil. It is well-known that the corresponding root $[\lambda,\mu]$ of the discriminant equation is of higher multiplicity.  If $V(\phi_1) = S([v])$ and $V(\phi_2) = S([w])$ are nonsingular complex spheres, then the pencil $V(\lambda q_1 +\mu q_2)$ corresponds to the line $\overline{x,y}$ in $\bbP^{n+1}$ spanned by the points $x$ and $y$.   A point $[\lambda v+\mu w]$ on the line defines a singular quadric if and only if 
$$D(\lambda v+\mu w) = (\lambda \alpha_{0}+\mu \beta_{0})^{n-1}q_0(\lambda v+\mu  w) = 0.$$
Our condition for quadrics $V(\phi_1)$ and $V(\phi_2)$ to be tangent to each other is that 
the equation $q_0(\lambda v+\mu  w)= 0$ has a double root. We have
$$q_0(\lambda v+\mu  w) = \lambda^2\la v,v\ra+2\lambda\mu \la v,w\ra +\mu^2\la w,w\ra.$$
Thus the condition becomes 
$$\det \begin{pmatrix} \la v,v\ra&\la v,w\ra\\
\la w,v\ra&\la w,w\ra\end{pmatrix} = 0.$$
\end{proof}

\begin{proposition} $n+1$ complex spheres $S([v_i])$ in $\bbP^n$  have a common point  if and only if $$\det \begin{pmatrix} \la w_1,w_1\ra&\ldots&\la w_1,w_{n+1}\ra\\
\la w_2,w_1\ra&\ldots&\la w_2,w_{n+1}\\
\vdots&\vdots&\vdots\\
\la w_{n+1},w_1\ra&\ldots&\la w_{n+1},w_{n+1}\ra\end{pmatrix} = 0,$$
where $w_i$ are the vectors of coordinates of the polar hyperplane of $[v_i]$. 
\end{proposition}

\begin{proof} We use the following known identity in the theory of determinants (see, for example, \cite{CAG}, Lemma 10.3.2). Let $A = (a_{ij}), B = (b_{ij})$ be two matrices of sizes $k\times m$ and $m\times k$ with $k\le m$. Let $|A_I|, |B_I|, I = (i_1,\ldots, i_k), 1\le i_1< \ldots <i_k \le m,$ be maximal minors of $A$ and $B$. Then
\beq\label{det}
|A\cdot B| = \sum_{I}|A_I||B_I|.
\eeq
Let $H_i: = \sum_{j=0}^{n+1} a_{j}^{(i)}t_j = 0$ be the polar hyperplanes of the complex spheres. We may assume that they are linearly independent, i.e. the vectors $w_i$ are linearly independent in $V$. Otherwise the determinant is obviously equal to zero. Thus the hyperplanes intersect at one point. The spheres  have a common point if and only if the intersection point of the  hyperplanes $H_i$ lies on the fundamental quadric. Let $X$ be the matrix with rows equal to vectors $w_i = (a_0^{(i)},\ldots,a_{n+1}^{(i)})$.  The projective coordinates of the intersection point are $[C_1,-C_2,\ldots,(-1)^{n+1}C_{n+2}]$,  where $C_j$ is the maximal minor obtained from $X$ by 
deleting the $j$-th column. Let 
$G$ be  the symmetric matrix defining the fundamental quadric. We take in the above  formula $A = X\cdot G, B = {}^tX$. Then the product $A\cdot B$ is equal to the left-hand side of the formula in the assertion of the proposition. The right-hand side is equal to $\pm (C_1C_{n+2}-\sum_{i=2}^{n+1} C_i^2)$. It is equal to zero if and only if the intersection point lies on the fundamental quadric.
\end{proof}

We refer to \cite{Study} for many other mutual geometrical properties  of circles expressed in terms of invariants of the orthogonal group $\Or(4)$.

\section{Real points}
We choose $V$ to be a real vector space equipped with a positive definite inner product $\la -, -\ra$. A real point in $\bbP(V)$ is represented by a nonzero vector $v\in V$. Since $\la v,v\ra > 0$, we obtain from Propositions \ref{P3.3} and \ref{P3.4}  that all real point sets $(x_1,\ldots,x_m)$ are stable points.
Another nice feature of real point sets is the criterion for vanishing of the Gram functions: $\det G(v_1,\ldots,v_k) = 0$ if and only if $v_1,\ldots,v_k$ are linear dependent vectors in $V$. 

It follows from the FFT and the SFT that the varieties $\bfO_n^m$ are defined over $\bbQ$. In particularly, we may speak about the set $\bfO_N^m(\bbR)$ of their real points.

\begin{theorem}\label{T7.1} Let $V$ be a real inner-product space. Consider the open subset $U$ of linear independent point sets $(x_1,\ldots,x_m)$ in $\bbP(V)(\bbR)$. Then the map
$$U\to \bfO_n^m(\bbR)$$
is injective.
\end{theorem}

\begin{proof} To show the injectivity of the map it suffices to show that 
\beq\label{orbit1}
(g(x_1),\ldots,g(x_m)) = 
(y_1,\ldots,y_m)
\eeq
 for $g\in \PO(V_\bbC)$ implies that $(g(x_1),\ldots,g(x_m)) = 
(y_1,\ldots,y_m)$ for some $g'\in \PO(V)$.  Choose an orthonormal basis in $V$ to identify $V$ with the Euclidean real space $\bbR^{n+1}$. The transformation  $g$ is represented by a complex orthogonal matrix.  If \eqref{orbit1} holds, we can find some representatives $v_i$ and $w_i$ of points $x_i,y_i$, respectively, and a matrix $A\in \Or(V_\bbC)$ such that 
$A\cdot v_i = w_i, i = 1,\ldots,m$. This is an inhomogeneous  system of linear equations in the entries of $A$. Since the rank of the matrix $[v_1,\ldots, v_n]$ with columns $v_i$ is maximal, there is a unique solution for $A$ and it is real. Thus $g$ is represented by a transformation from $\Or(n+1,\bbR)$.
\end{proof}
 
 Let us look at the rational invariants $R_{a_1,\ldots,a_k}$. Let $\phi_{ij}, \pi-\phi_{ij},$ denote the angles between basis vectors of the  lines $x_i = \bbR v_i$.  Obviously, 
 $$R_{ij} = \cos^2 \phi_{ij}$$ is well defined and does not depend on the choice of the bases. Also,
  $$R_{ij\ldots k} = \cos \phi_{ij}\cdots \cos \phi_{ki}$$
  are well defined too. Applying the previous theorem, we see that the cyclic products of the cosines determine uniquely the orbit of a linearly independent point set. 

Finally,  let us discuss configuration spaces of real spheres. For this we have to choose $V = \bbR^{n+2}$ to be a real space with quadratic form $q_0$ of signature $(n+1,1)$ defined in \eqref{fq}. A real  sphere with non-zero radius is defined by formula \eqref{sphere3}, where the coefficients $(\alpha_0,\ldots,\alpha_{n+1})$  belong to the set 
$q_0^{-1}(\bbR_{> 0})$.   It consists of two connected components corresponding to the choice of the sign of $\alpha_0$. Choose the component $V^+$ where $\alpha_0 > 0$. The image $V^+/\bbR_{>0}$ of $V^+$ in the projective space $\bbP(V^+)$ is, by definition, the \emph{hyperbolic space} $\bbH^{n+1}$. Each point in $\bbH^{n+1}$ can be uniquely represented by a unique vector $v= (\alpha_0,\ldots,\alpha_{n+1})$ with 
$$\la v,v\ra  = \sum_{i=1}^{n}\alpha_i^2+2\alpha_0\alpha_{n+1} = 1, \quad \alpha_0 > 0.$$
Each $v\in \bbH^{n+1}$ defines the orthogonal hyperplane 
$$H_v = \{x\in \bbH^{n+1}:\la x,v\ra = 0, q(x) = 1\}.$$ 
The cosine of the angle between the hyperplanes $H_v$ and $H_w$ is defined by 
$$\cos \phi = -\la v,w\ra.$$ 
If $|\la v,w\ra|  > 1$, the hyperplanes are divergent, i.e. they do not intersect in the hyperbolic space.  In this case $\cosh (|\la v,w\ra|)$ is equal to the distance between the two divergent hyperplanes. If $\la v,w\ra = 1$, the hyperplanes are  parallel. By Proposition \ref{P3.1}, the corresponding real spheres $S(v)$ and $S(w)$ are tangent to each other. 

%After the linear change of coordinates 
%$$\alpha_0 = \beta_0+\beta_{n+1}, \alpha_{n+1} = \beta_0-\beta_{n+1}, \alpha_i = \beta_i, i\ne 0,n+1,$$ we obtain that each point can be uniquely represented by the vector $(\beta_0,\ldots,\beta_{n+1}$ with $\beta_0 = 1, \beta_{n+1} > 0$ and 
%$$\sum_{i=1}^n\beta_i^2 +\beta_{n+1}^2 > 1.$$
%This is one of conformal models of the hyperbolic space.  One more change of variables (not linear this time), 
%$\gamma_i = \beta_i(\sum_{i=1}^{n+1}\beta_i^2)$ allows one to represent a point in $\bbH^{n+1}$ by a point $(\gamma_1,\ldots,\gamma_{n+1})$ by an interior points of the $n+1$-dimensional unit ball in $\bbR^{n+1}$:
%$$B^{n+1}:\sum_{i=1}^{n+1}y_i^2 < 1.$$
%This is the familiar Klein's model of the hyperbolic  space. 

%The Riemann metric in this model is given by 
%$$ds^2 = 4(1-\sum_{i=1}^{n+1}y_i^2)^{-2}\sum_{i=1}^{n+1}dy_i.$$

\end{document}